\documentclass[12pt, a4paper]{amsart} % <<<
\usepackage{amssymb,amsfonts,amsmath, psfrag,eepic,colordvi,graphicx,epsfig,ytableau, booktabs}
\usepackage{amssymb,latexsym, graphics, array}
\usepackage{MnSymbol}
\usepackage[enableskew]{youngtab}
\usepackage{float, enumerate,enumitem}
\usepackage{tikz,ifthen}
\usepackage{pgfplots}
\usetikzlibrary{math}
\topskip  
\numberwithin{equation}{section}
\newtheorem{theorem}{Theorem}[section]

\newtheorem{corollary}[theorem]{Corollary}

\newtheorem{conjecture}[theorem]{Conjecture}

\newtheorem{lemma}[theorem]{Lemma}

\usepackage{makecell} 
\usepackage[numbers,sort&compress]{natbib}
\usepackage{ytableau}
\usepackage{textcomp}
\usepackage{xcolor}
\usetikzlibrary{backgrounds}
\usetikzlibrary{positioning, matrix}

\newcommand{\blackcirclenumber}[1]{%
    \tikz[baseline=(char.base)]{%
        \node[shape=circle, fill=black, inner sep=1.2pt, 
              minimum size=1em, text=white] (char) {\sffamily\bfseries\footnotesize #1};
    }%
}

\DeclareMathOperator{\ASM}{ASM}
\DeclareMathOperator{\HSASM}{HSASM}

\title[Doubly refined enumeration of Alternating Sign Matrices]{Yet another doubly refined enumeration of Alternating Sign Matrices}

\author{Guo-Niu Han}
\address{I.R.M.A., UMR 7501, Universit\'e de Strasbourg
et CNRS, 7 rue Ren\'e Descartes, F-67084 Strasbourg, France}
\email{guoniu.han@unistra.fr}

\author[Lihong Yang]{Lihong Yang $^*$}
\address{Institute for Advanced Study in Mathematics,
Harbin Institute of Technology,
Heilongjiang 150001, P.R. China}
\email{lihong.yang.math@gmail.com}
\address{I.R.M.A., UMR 7501, Universit\'e de Strasbourg
et CNRS, 7 rue Ren\'e Descartes, F-67084 Strasbourg, France}
\email{lihong.yang@unistra.fr}
\thanks{$^*$ Lihong Yang is the corresponding author.}



\date{January 6, 2026}

\begin{document}

\begin{abstract} % <<<
Since the alternating sign matrix conjecture, proposed by Mills, Robbins, and Rumsey in 1982, was proved by Zeilberger and Kuperberg, several refined enumerations have been considered.  
In particular, Behrend et al. obtained a quadruply refined enumeration by adding certain parameters.
In this paper, we revisit the doubly refined enumeration of alternating sign matrices
by adding three parameters: the number of $-1$'s, the position of the $1$ in the first row, 
and the position of the $1$ in the last row. 
Using Lascoux's formula on symmetry functions, we derive a new determinantal formula for this doubly refined enumeration. 

Besides the enumeration conjecture,  Mills et al. also proposed a decomposition conjecture, which was subsequently proven by Kuperberg.
We present a refinement of that decomposition conjecture.

\end{abstract}
% >>>

\maketitle

\section{Introduction}  % <<<

An {\em alternating sign matrix} (ASM) of order $n$ is an $n \times n$ square matrix whose entries are either $0$, $1$, or $-1$. In such a matrix, the nonzero entries in each row and each column alternate between $1$ and $-1$, beginning and ending with $1$. For example, the following matrix
\begin{equation}\label{m:asm1}
\left(
\begin{array}{ccccc}
0 & 1 & 0 & 0 & 0 \\
1 & -1 & 0 & 1 & 0 \\
0 & 1 & 0 & -1 & 1 \\
0 & 0 & 0 & 1 & 0 \\
0 & 0 & 1 & 0 & 0
\end{array}
\right)
\end{equation}
is an ASM of order $5$.

Mills, Robbins, and Rumsey first discovered the concept of ASMs in their study of Dodgson’s method for computing determinants \cite{MRR1, MRR2, RR}.
They investigated the enumeration of ASMs and conjectured that the number of $n\times n$ ASMs is given by
\begin{equation}\label{asm}
    \prod_{j=1}^{n}\frac{(3j-2)!}{(n+j-1)!}.
\end{equation}
This is the famous Alternating Sign Matrix conjecture. 
In 1996, Zeilberger~\cite{Ze96a} provided the first proof of this conjecture by demonstrating the equivalence between ASMs and symmetric, self-complementary plane partitions. 
Later, Kuperberg~\cite{Ku96} gave a significantly different proof, using 
the six-vertex model and a determinantal formula due to Izergin and Korepin~\cite{Iz87, KBI93}.
Kuperberg's proof marked a major milestone in the study of ASMs and has inspired decades of further research in the field \cite{
Fi07, FK20a, FK20b, FK22, KBI93}.

A refined enumeration is another key direction in the study of ASMs, motivated by the observation that the first and last rows, as well as the leftmost and rightmost columns,
each contain exactly one $1$ and no $-1$s.
The positions of these unique boundary $1$s provide natural discrete parameters to refine the enumeration of ASMs.
Zeilberger first studied this refinement \cite{Ze96b}. By extending Kuperberg’s method, he proved that the number of $ n\times n$ ASMs whose unique $1$ in the first row lies in column $r$ is given by 
$$
     \frac{(r)_{n-1}(n+1-r)_{n-1}}{(n-1)!}\prod_{j=1}^{n}\frac{(3j-2)!}{(n+j-1)!}.
$$
Here $(a)_{n}=\prod_{i=0}^{n-1}(a+i)$. 
This result confirmed the conjecture of Mills, Robbins, and Rumsey~\cite{MRR1, MRR2}.

Subsequent work produced many results on doubly, triply, and quadruply refined enumerations, in which the positions of entries equal to $1$'s are fixed on boundary rows and/or columns.
Mills, Robbins, and Rumsey first introduced the double refinement enumeration \cite{MRR3}, refining ASMs by the positions of the $1$ in the top and bottom rows.
However, they did not obtain an explicit formula.
This gap was later filled by Fonseca and Zinn-Justin~\cite{FZ08} and Stroganov~\cite{St06}.
Related problems on doubly refined enumeration were further studied in \cite{FR09, Fi11, KR10, BEZ13}.
Progress on triply and quadruply refined enumerations appeared in \cite{AR13, Beh13, BEZ12}.
Behrend, Francesco, and Zinn-Justin~\cite{BEZ13}
also introduced two additional bulk statistics: the number of $-1$s and the inversion number of ASMs, and showed that the refined generating functions can be expressed as the determinant of a certain $n\times n$ matrix.

\medskip
Parallel to these works, we pursue the study of the 
explicit generating function for the doubly-refined enumeration of alternating sign matrices with respect to the number of $-1$'s
using a new approach. Our contributions span the following three directions:

\begin{enumerate}
    \item[(1)] We follow Kuperberg's method to derive the generating function for alternating sign matrices.  
    However,  our approach employs Lascoux's determinantal formula involving symmetric functions,  
    thereby avoiding the introduction of the superfluous variable $\epsilon$.
    \item[(2)] We obtain an explicit generating function for the doubly-refined enumeration of alternating sign matrices with respect to the number of $-1$'s. 
    Our formula (Theorem \ref{th:main}) is noticeably different from that of Behrend et al.
    \item[(3)] 
Although our final expression may appear complex, it gives rise to simple determinantal identities, such as Identity \eqref{eq:n!}, which seem to lack simple direct proofs.
\end{enumerate}

We now state our main results after introducing some notation. 
For each positive integer $n$, let $\ASM(n)$ denote the set of all $n\times n$ ASMs. 
The first of these sets is reproduced below:
$$
\begin{aligned}
&\ASM(1) = \bigl\{ (1) \bigr\}, \\
&\ASM(2) = \left\{
\begin{pmatrix} 1 & 0 \\ 0 & 1 \end{pmatrix},\;
\begin{pmatrix} 0 & 1 \\ 1 & 0 \end{pmatrix}
\right\}, \\
&\ASM(3) = \left\{
\begin{aligned}
& \begin{pmatrix} 1 & 0 & 0 \\ 0 & 1 & 0 \\ 0 & 0 & 1 \end{pmatrix},\quad
\begin{pmatrix} 0 & 1 & 0 \\ 1 & 0 & 0 \\ 0 & 0 & 1 \end{pmatrix},\quad
\begin{pmatrix} 1 & 0 & 0 \\ 0 & 0 & 1 \\ 0 & 1 & 0 \end{pmatrix},\quad
\begin{pmatrix} 0 & 1 & 0 \\ 0 & 0 & 1 \\ 1 & 0 & 0 \end{pmatrix}, \\
& \begin{pmatrix} 0 & 0 & 1 \\ 1 & 0 & 0 \\ 0 & 1 & 0 \end{pmatrix},\quad
\begin{pmatrix} 0 & 0 & 1 \\ 0 & 1 & 0 \\ 1 & 0 & 0 \end{pmatrix},\quad
\begin{pmatrix} 0 & 1 & 0 \\ 1 & -1 & 1 \\ 0 & 1 & 0 \end{pmatrix}
\end{aligned}
\right\}.
\end{aligned}
$$
For each $A \in \ASM(n)$, define the following statistics:
\begin{align*}
	\mu(A) &:  \text{the number of $-1$s in $A$}; \\
	f(A) &:  \text{the column index of the unique $1$ in the first row of $A$}; \\
	\ell(A) &: \text{the column index of the unique $1$ in the last row of $A$}.
\end{align*}
For example, consider 
$$A=\begin{pmatrix}
0 & 1 & 0 \\
1 & -1 & 1 \\
0 & 1 & 0
\end{pmatrix}\in \ASM(3),$$ 
then $\mu(A)=1,f(A)=2, \ell(A)=2$. 

\medskip

We  derive an explicit  generating function for the doubly-refined enumeration of alternating sign matrices with respect to the number of $-1$s
\begin{equation}\label{def:An}
A_n(z,\rho,\tau) = \sum_{A\in \ASM(n)} z^{\mu(A)} \rho^{f(A)} \tau^{\ell(A)}.
\end{equation}
In practice, the above expression takes a much simpler form when setting $z=2+q+q^{-1}$, as stated in the following theorem.
\begin{theorem}\label{th:main}
The generating function  for 
the $(2+q+q^{-1})$-enumeration of doubly-refined ASMs is
\begin{align}\label{eq:An:main}
    A_n\bigl(2+q + q^{-1}, \rho, \tau\bigr)
        &= \left( \frac{(1 + \rho q)(\tau + q)}{(1 + q)^2} \right)^{\!n - 1} 
					 \cdot \frac{\tau \rho \cdot \det{\bigl(K^{\rho,\tau}_{i,j}(q)\bigr)_{1 \le i,j \le n}}}{(2 - q - q^{-1})^{\frac{n(n - 1)}{2}}},
\end{align}
    where
\begin{align*}
	K^{\rho,\tau}_{i,j}(q) = &\sum_{k=-n+1}^{n} (-1)^{k+1} \frac{q^k - q^{-k}}{q - q^{-1}} 
    \binom{2j + k - 3}{i - 1} \left[ \binom{n - 2}{k + j - 3} \right.  \\
    & \left. + \left( \frac{\rho + q}{1 + \rho q} + \frac{1 + \tau q}{\tau + q} \right) \binom{n - 2}{k + j - 2} 
     + \frac{(\rho + q)(1 + \tau q)}{(1 + \rho q)(\tau + q)} \binom{n - 2}{k + j - 1} \right].
\end{align*}      
    
\end{theorem}

While the derived expression may seem intricate, it nevertheless leads to simple determinantal identities, which seem to lack simple direct proofs.
For example, setting $q=-1, \tau=\rho=1$ in Theorem~\ref{th:main}, we obtain the following determinantal identity, 
since the number $A_n(0,1,1)$ enumerates the $n \times n$ permutation matrices.
\begin{corollary}\label{cor:n!}
Let $n$ be a positive integer. We have 
\begin{equation}\label{eq:n!}
	\det{\bigl(K^{1,1}_{i,j}(-1)\bigr)_{1 \le i,j \le n}}=4^{\frac{n(n-1)}{2}}n!,
\end{equation}
where
$$
	K^{1,1}_{i,j}(-1)=\sum_{k=-n+1}^{n}k\binom{2j+k-3}{i-1}\binom{n}{j+k-1}.
$$
\end{corollary}

This paper is organized as follows.
In Section~\ref{sec:model}, we review the fundamental connection between ASMs and the six-vertex model from statistical mechanics. 
Fortunately, the explicit generating function for the latter model was established by Izergin and Korepin. 
In Section~\ref{sec:Lascoux}, we recall a determinantal formula involving symmetric functions due to Lascoux.
The combination of Lascoux's formula with the Izergin–Korepin theorem plays a crucial role in the proof of Theorem~\ref{th:main},
which we develop fully in Section \ref{sec:proof_mth}.
As an application, we derive several determinantal formulas in Section~\ref{sec:applications}.
Mills et al. proposed a decomposition conjecture, which was subsequently proven by Kuperberg.
In Section~\ref{sec:comparison}, we compare our results with previously known results. 
We present a refinement of that decomposition conjecture in Section~\ref{sec:conj}.

% >>>

\section{Six-vertex model}\label{sec:model}  % <<<
There is a bijection between the six-vertex model and alternating sign matrices. In this section, we introduce this bijection and review the work of Kuperberg~\cite{Ku96}.
The six-vertex model is concerned with the multiplicative weighted enumeration of arrow configurations, commonly referred to as states, on the edges of a planar tetravalent graph $G$ (see Fig.~\ref{fig:md2}). 
At each vertex, four edges meet, and the arrows must satisfy the "zero divergence" rule: two arrows point in and two point out.
The six allowed arrow configurations around a vertex are called vertex states and are usually labeled $1$ to $6$. 
For clarity, we denote them by $\blackcirclenumber{1}$ to $\blackcirclenumber{6}$, as shown in Fig.~\ref{fig:md2}.
The six-vertex model on an $n\times n$ square grid with the following boundary conditions is also called {\em square ice}~\cite{Ku96}: 
The left and right boundary edges point inward, and the top and bottom boundary edges point outward (see Fig.~\ref{fig:bij}).
Elkies et al. \cite{EKLP} established an explicit bijection between square ice configurations and $n\times n$ ASMs,
which maps each state to a matrix entry, as shown in Fig.~\ref{fig:md2}. 
The inverse of that bijection is denoted by $\varphi: A \mapsto G$, which maps an ASM $A$ to a square ice configuration $G$.
Fig.~\ref{fig:bij} illustrates the square ice configuration corresponding to the ASM in (\ref{m:asm1}).

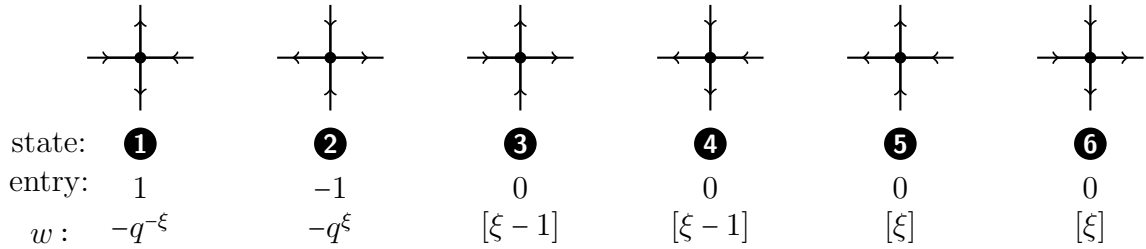
\begin{figure} [h] % <<<
 	\centering
\begin{tikzpicture}
 		\draw[>-<, thick] (0,0)--(1,0); 
            \draw[-, thick] (-0.2,0)--(1.2,0);
            \draw[<->, thick] (0.5,-0.5)--(0.5,0.5);
            \draw[-, thick] (0.5,-0.7)--(0.5,0.7);
            \coordinate [label=above:$\blackcirclenumber{1}$] (x) at (0.5,-1.5);
            \filldraw[fill=black](0.5,0)circle(0.07);
           \coordinate [label=above:$1$] (x) at (0.5,-2);
           \coordinate [label=above:$-q^{-\xi}$] (x) at (0.5,-2.6);
           
           \draw[<->, thick] (2.5,0)--(3.5,0);
           \draw[-, thick] (2.3,0)--(3.7,0);
           \draw[>-<, thick] (3,-0.5)--(3,0.5);
           \draw[-, thick] (3,-0.7)--(3,0.7);
           \coordinate [label=above:$\blackcirclenumber{2}$] (x) at (3,-1.5);
           \filldraw[fill=black](3,0)circle(0.07);
           \coordinate [label=above:$-1$] (x) at (3,-2);
           \coordinate [label=above:$-q^{\xi}$] (x) at (3,-2.6);
           
           \draw[>->, thick] (5,0)--(6,0); 
           \draw[-, thick] (4.8,0)--(6.2,0);
           \draw[>->, thick] (5.5,-0.5)--(5.5,0.5);
           \draw[-, thick] (5.5,-0.7)--(5.5,0.7);
           \coordinate [label=above:$\blackcirclenumber{3}$] (x) at (5.5,-1.5);
           \filldraw[fill=black](5.5,0)circle(0.07);
           \coordinate [label=above:$0$] (x) at (5.5,-2);
           \coordinate [label=above:$\mbox{[$\xi-1$]}$] (x) at (5.5,-2.6);
           
           \draw[<-<, thick] (7.5,0)--(8.5,0); 
           \draw[-, thick] (7.3,0)--(8.7,0);
           \draw[<-<, thick] (8,-0.5)--(8,0.5);
           \draw[-, thick] (8,-0.7)--(8,0.7);
           \coordinate [label=above:$\blackcirclenumber{4}$] (x) at (8,-1.5);
           \filldraw[fill=black](8,0)circle(0.07);
           \coordinate [label=above:$0$] (x) at (8,-2);
           \coordinate [label=above:$\mbox{[$\xi-1$]}$] (x) at (8,-2.6);
           
           \draw[<-<, thick] (10,0)--(11,0); 
           \draw[-, thick] (9.8,0)--(11.2,0);
           \draw[>->, thick] (10.5,-0.5)--(10.5,0.5);
           \draw[-, thick] (10.5,-0.7)--(10.5,0.7);
           \coordinate [label=above:$\blackcirclenumber{5}$] (x) at (10.5,-1.5);
           \filldraw[fill=black](10.5,0)circle(0.07);
           \coordinate [label=above:$0$] (x) at (10.5,-2);
           \coordinate [label=above:$\mbox{[$\xi$]}$] (x) at (10.5,-2.6);
           
           \draw[>->, thick] (12.5,0)--(13.5,0); 
           \draw[-,  thick] (12.3,0)--(13.7,0);
           \draw[<-<,  thick] (13,-0.5)--(13,0.5);
           \draw[-,  thick] (13,-0.7)--(13,0.7);
        \coordinate [label=above:$\blackcirclenumber{6}$] (x) at (13,-1.5);
        \filldraw[fill=black](13,0)circle(0.07);
        \coordinate [label=above:$0$] (x) at (13,-2);
        \coordinate [label=above:$\mbox{[$\xi$]}$] (x) at (13,-2.6);

        \coordinate [label=above:$\mbox{state:}$] (x) at (-0.7,-1.4);
        \coordinate [label=above:$\mbox{entry:}$] (x) at (-0.7,-2);
        \coordinate [label=above:$w:$] (x) at (-0.7,-2.6);
        
  \end{tikzpicture}
\caption{The six vertex states, their corresponding matrix entries, and associated weights.}\label{fig:md2}
 \end{figure} % >>> 

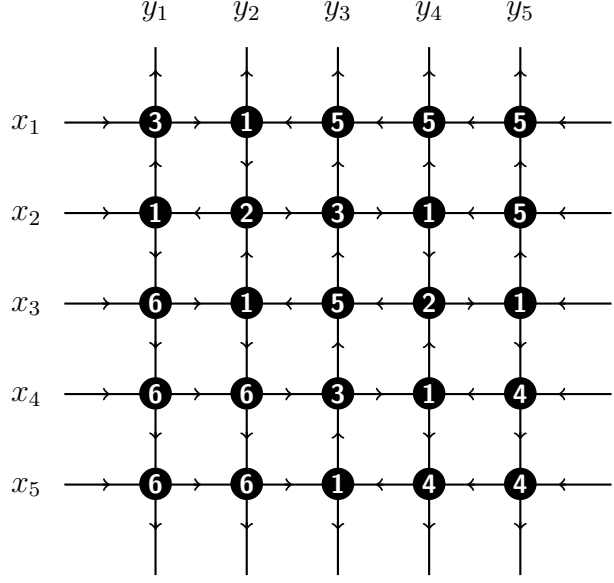
\begin{figure} % <<<
    \centering
   \begin{tikzpicture}

   \coordinate [label=above:$\blackcirclenumber{6}$] (x) at (1.2,-0.35);
   \coordinate [label=above:$\blackcirclenumber{6}$] (x) at (2.4,-0.35);
   \coordinate [label=above:$\blackcirclenumber{1}$] (x) at (3.6,-0.35);
   \coordinate [label=above:$\blackcirclenumber{4}$] (x) at (4.8,-0.35);
   \coordinate [label=above:$\blackcirclenumber{4}$] (x) at (6,-0.35);

  \coordinate [label=above:$\blackcirclenumber{6}$] (x) at (1.2,0.85);
   \coordinate [label=above:$\blackcirclenumber{6}$] (x) at (2.4,0.85);
   \coordinate [label=above:$\blackcirclenumber{3}$] (x) at (3.6,0.85);
   \coordinate [label=above:$\blackcirclenumber{1}$] (x) at (4.8,0.85);
   \coordinate [label=above:$\blackcirclenumber{4}$] (x) at (6,0.85);

   \coordinate [label=above:$\blackcirclenumber{6}$] (x) at (1.2,2.05);
   \coordinate [label=above:$\blackcirclenumber{1}$] (x) at (2.4,2.05);
   \coordinate [label=above:$\blackcirclenumber{5}$] (x) at (3.6,2.05);
   \coordinate [label=above:$\blackcirclenumber{2}$] (x) at (4.8,2.05);
   \coordinate [label=above:$\blackcirclenumber{1}$] (x) at (6,2.05);

  \coordinate [label=above:$\blackcirclenumber{1}$] (x) at (1.2,3.25);
   \coordinate [label=above:$\blackcirclenumber{2}$] (x) at (2.4,3.25);
   \coordinate [label=above:$\blackcirclenumber{3}$] (x) at (3.6,3.25);
   \coordinate [label=above:$\blackcirclenumber{1}$] (x) at (4.8,3.25);
   \coordinate [label=above:$\blackcirclenumber{5}$] (x) at (6,3.25);

   \coordinate [label=above:$\blackcirclenumber{3}$] (x) at (1.2,4.45);
   \coordinate [label=above:$\blackcirclenumber{1}$] (x) at (2.4,4.45);
   \coordinate [label=above:$\blackcirclenumber{5}$] (x) at (3.6,4.45);
   \coordinate [label=above:$\blackcirclenumber{5}$] (x) at (4.8,4.45);
   \coordinate [label=above:$\blackcirclenumber{5}$] (x) at (6,4.45);

   \coordinate [label=above:$x_1$] (x) at (-0.5,4.5);
   \coordinate [label=above:$x_2$] (x) at (-0.5,3.3);
   \coordinate [label=above:$x_3$] (x) at (-0.5,2.1);
   \coordinate [label=above:$x_4$] (x) at (-0.5,0.9);
   \coordinate [label=above:$x_5$] (x) at (-0.5,-0.3);

   \coordinate [label=above:$y_1$] (x) at (1.2,6);
   \coordinate [label=above:$y_2$] (x) at (2.4,6);
   \coordinate [label=above:$y_3$] (x) at (3.6,6);
   \coordinate [label=above:$y_4$] (x) at (4.8,6);
   \coordinate [label=above:$y_5$] (x) at (6,6);

   \draw[-, thick] (0,0)--(1,0);
   \draw[-, thick] (1.4,0)--(2.2,0);
    \draw[-, thick] (2.6,0)--(3.4,0);
    \draw[-, thick] (3.8,0)--(4.6,0);
    \draw[-, thick] (5,0)--(5.8,0);
    \draw[-, thick] (6.2,0)--(7.2,0);

     \draw[-, thick] (0,1.2)--(1,1.2);
   \draw[-, thick] (1.4,1.2)--(2.2,1.2);
    \draw[-, thick] (2.6,1.2)--(3.4,1.2);
    \draw[-, thick] (3.8,1.2)--(4.6,1.2);
    \draw[-, thick] (5,1.2)--(5.8,1.2);
    \draw[-, thick] (6.2,1.2)--(7.2,1.2);

    \draw[-, thick] (0,2.4)--(1,2.4);
   \draw[-, thick] (1.4,2.4)--(2.2,2.4);
    \draw[-, thick] (2.6,2.4)--(3.4,2.4);
    \draw[-, thick] (3.8,2.4)--(4.6,2.4);
    \draw[-, thick] (5,2.4)--(5.8,2.4);
    \draw[-, thick] (6.2,2.4)--(7.2,2.4);

    \draw[-, thick] (0,3.6)--(1,3.6);
   \draw[-, thick] (1.4,3.6)--(2.2,3.6);
    \draw[-, thick] (2.6,3.6)--(3.4,3.6);
    \draw[-, thick] (3.8,3.6)--(4.6,3.6);
    \draw[-, thick] (5,3.6)--(5.8,3.6);
    \draw[-, thick] (6.2,3.6)--(7.2,3.6);

    \draw[-, thick] (0,4.8)--(1,4.8);
   \draw[-, thick] (1.4,4.8)--(2.2,4.8);
    \draw[-, thick] (2.6,4.8)--(3.4,4.8);
    \draw[-, thick] (3.8,4.8)--(4.6,4.8);
    \draw[-, thick] (5,4.8)--(5.8,4.8);
    \draw[-, thick] (6.2,4.8)--(7.2,4.8);

    \draw[-, thick] (1.2,-1.2)--(1.2,-0.2);
    \draw[-, thick] (1.2,0.2)--(1.2,1);
    \draw[-, thick] (1.2,1.4)--(1.2,2.2);
    \draw[-, thick] (1.2,2.6)--(1.2,3.4);
    \draw[-, thick] (1.2,3.8)--(1.2,4.6);
    \draw[-, thick] (1.2,5)--(1.2,5.8);

     \draw[-, thick] (2.4,-1.2)--(2.4,-0.2);
    \draw[-, thick] (2.4,0.2)--(2.4,1);
    \draw[-, thick] (2.4,1.4)--(2.4,2.2);
    \draw[-, thick] (2.4,2.6)--(2.4,3.4);
    \draw[-, thick] (2.4,3.8)--(2.4,4.6);
    \draw[-, thick] (2.4,5)--(2.4,5.8);

     \draw[-, thick] (3.6,-1.2)--(3.6,-0.2);
    \draw[-, thick] (3.6,0.2)--(3.6,1);
    \draw[-, thick] (3.6,1.4)--(3.6,2.2);
    \draw[-, thick] (3.6,2.6)--(3.6,3.4);
    \draw[-, thick] (3.6,3.8)--(3.6,4.6);
    \draw[-, thick] (3.6,5)--(3.6,5.8);

     \draw[-, thick] (4.8,-1.2)--(4.8,-0.2);
    \draw[-, thick] (4.8,0.2)--(4.8,1);
    \draw[-, thick] (4.8,1.4)--(4.8,2.2);
    \draw[-, thick] (4.8,2.6)--(4.8,3.4);
    \draw[-, thick] (4.8,3.8)--(4.8,4.6);
    \draw[-, thick] (4.8,5)--(4.8,5.8);

     \draw[-, thick] (6,-1.2)--(6,-0.2);
    \draw[-, thick] (6,0.2)--(6,1);
    \draw[-, thick] (6,1.4)--(6,2.2);
    \draw[-, thick] (6,2.6)--(6,3.4);
    \draw[-, thick] (6,3.8)--(6,4.6);
    \draw[-, thick] (6,5)--(6,5.8);

    \draw[->, thick] (0.5,0)--(0.6,0);
    \draw[->, thick] (1.7,0)--(1.8,0);
    \draw[->, thick] (2.9,0)--(3,0);
    \draw[<-, thick] (4.1,0)--(4.2,0);
    \draw[<-, thick] (5.3,0)--(5.4,0);
    \draw[<-, thick] (6.5,0)--(6.6,0);

     \draw[->, thick] (0.5,1.2)--(0.6,1.2);
    \draw[->, thick] (1.7,1.2)--(1.8,1.2);
    \draw[->, thick] (2.9,1.2)--(3,1.2);
    \draw[>-, thick] (4.1,1.2)--(4.2,1.2);
    \draw[<-, thick] (5.3,1.2)--(5.4,1.2);
    \draw[<-, thick] (6.5,1.2)--(6.6,1.2);

     \draw[->, thick] (0.5,2.4)--(0.6,2.4);
    \draw[->, thick] (1.7,2.4)--(1.8,2.4);
    \draw[<-, thick] (2.9,2.4)--(3,2.4);
    \draw[<-, thick] (4.1,2.4)--(4.2,2.4);
    \draw[->, thick] (5.3,2.4)--(5.4,2.4);
    \draw[<-, thick] (6.5,2.4)--(6.6,2.4);

     \draw[->, thick] (0.5,3.6)--(0.6,3.6);
    \draw[<-, thick] (1.7,3.6)--(1.8,3.6);
    \draw[->, thick] (2.9,3.6)--(3,3.6);
    \draw[->, thick] (4.1,3.6)--(4.2,3.6);
    \draw[<-, thick] (5.3,3.6)--(5.4,3.6);
    \draw[<-, thick] (6.5,3.6)--(6.6,3.6);

     \draw[->, thick] (0.5,4.8)--(0.6,4.8);
    \draw[->, thick] (1.7,4.8)--(1.8,4.8);
    \draw[<-, thick] (2.9,4.8)--(3,4.8);
    \draw[<-, thick] (4.1,4.8)--(4.2,4.8);
    \draw[<-, thick] (5.3,4.8)--(5.4,4.8);
    \draw[<-, thick] (6.5,4.8)--(6.6,4.8);

    \draw[<-, thick] (1.2,-0.6)--(1.2,-0.5);
    \draw[<-, thick] (1.2,0.6)--(1.2,0.7);
    \draw[<-, thick] (1.2,1.8)--(1.2,1.9);
    \draw[<-, thick] (1.2,3)--(1.2,3.1);
    \draw[->, thick] (1.2,4.2)--(1.2,4.3);
    \draw[->, thick] (1.2,5.4)--(1.2,5.5);

    \draw[<-, thick] (2.4,-0.6)--(2.4,-0.5);
    \draw[<-, thick] (2.4,0.6)--(2.4,0.7);
    \draw[<-, thick] (2.4,1.8)--(2.4,1.9);
    \draw[->, thick] (2.4,3)--(2.4,3.1);
    \draw[<-, thick] (2.4,4.2)--(2.4,4.3);
    \draw[->, thick] (2.4,5.4)--(2.4,5.5);

    \draw[<-, thick] (3.6,-0.6)--(3.6,-0.5);
    \draw[->, thick] (3.6,0.6)--(3.6,0.7);
    \draw[->, thick] (3.6,1.8)--(3.6,1.9);
    \draw[->, thick] (3.6,3)--(3.6,3.1);
    \draw[->, thick] (3.6,4.2)--(3.6,4.3);
    \draw[->, thick] (3.6,5.4)--(3.6,5.5);

    \draw[<-, thick] (4.8,-0.6)--(4.8,-0.5);
    \draw[<-, thick] (4.8,0.6)--(4.8,0.7);
    \draw[->, thick] (4.8,1.8)--(4.8,1.9);
    \draw[<-, thick] (4.8,3)--(4.8,3.1);
    \draw[->, thick] (4.8,4.2)--(4.8,4.3);
    \draw[->, thick] (4.8,5.4)--(4.8,5.5);

    \draw[<-, thick] (6,-0.6)--(6,-0.5);
    \draw[<-, thick] (6,0.6)--(6,0.7);
    \draw[<-, thick] (6,1.8)--(6,1.9);
    \draw[->, thick] (6,3)--(6,3.1);
    \draw[->, thick] (6,4.2)--(6,4.3);
    \draw[->, thick] (6,5.4)--(6,5.5);

   \end{tikzpicture}
    \caption{ The square ice configuration corresponding to the ASM in (\ref{m:asm1}).}
    \label{fig:bij}
\end{figure} % >>>

For $ 1 \leq i \leq 6$, let $n_{i}$ denote the number of vertices in state $i$. 
By the properties of the square ice model, in any $n\times n$ configuration, the number of vertices in state $1$ exceeds that in state $2$ by exactly $n$, while the numbers of vertices in states $3$ and $4$ are equal, as are those in states $5$ and $6$. That is,
\begin{equation}\label{ni}
n_1 = n + n_2, \quad n_3 = n_4, \quad \text{and} \quad n_5 = n_6.
\end{equation}
For example, in the $5 \times 5$ configuration shown in Fig.~\ref{fig:bij}, we have $n_1=7$, $n_2=2$, $n_3=3$, $n_4=3$, $n_5=5$, and $n_6=5$.

Given two sets of  parameters $X=\{x_1,\dots,x_n\}$ and $Y=\{y_1,\dots,y_n\}$. 
For a square ice configuration $G$ and a vertex $v=(i,j)$ in $G$, 
we assign a weight $w(v):=w_s(v)=w_s(x_i,y_j)$ according to its state $s$ in $G$.  
Let $\xi = x_i - y_j$ and define 
$$
[\xi]=\frac{q^{\xi}-q^{-\xi}}{q-q^{-1}}.
$$
% \begin{equation}\label{def:operator}
% [\xi]=\frac{q^{\xi}-q^{-\xi}}{q-q^{-1}}.
% \end{equation}
The weights of the six vertex states are shown in the last line of Fig.~\ref{fig:md2}. 

The weight of a square ice configuration $G$ is the product of the weights of all vertices in  $G$:
$$
W(G)=\prod_{v \in G}w(v).
$$
For example,  in the $5 \times 5$ configuration shown in Fig.~\ref{fig:bij},  if $x_i=x$ and $y_i=y$ for all $1\le i\le 5$, then the weight is
$$
(-q^{-\xi})^7\cdot(-q^\xi)^2\cdot[\xi-1]^3\cdot[\xi-1]^3\cdot[\xi]^5\cdot[\xi]^5=-q^{-5\xi}\cdot[\xi-1]^6\cdot[\xi]^{10}.
$$

Let $\mathbb{S}_n$ be the set of all $n\times n$ square ice configurations.  
The partition function of the square ice configuration is given by
\begin{equation}\label{def:Zn}
Z_n(X,Y)=\sum_{G \in \mathbb{S}_n} W(G)=\sum_{G \in \mathbb{S}_n} \prod_{v \in G}w(v)
=\sum_{G \in \mathbb{S}_n} \prod_{v \in G}w_s(x_i,y_j),
\end{equation}
where the vertex $v$ is located at position $(i,j)$ in $G$ and is in state $s$.
Izergin~\cite{Iz87} and Korepin~\cite{KBI93} obtained an explicit formula for $Z_n(X,Y)$,
  which plays a key role in proving our main results.

\begin{theorem}[Izergin, Korepin]\label{ikthv1} 
We have   
\begin{equation*}
        Z_n(X,Y)=\frac{(-1)^n(\prod_{i=1}^{n}q^{y_i-x_i})\prod_{1 \leq i,j\leq n}[x_i-y_j][x_i-y_j-1]}{(\prod_{1\leq j <i \leq n}[x_i-x_j])(\prod_{1 \leq i <j \leq n}[y_i-y_j])}\det{\bigl(M_{i,j}\bigr)_{1 \le i,j \le n}},
    \end{equation*}
    where
    $$M_{i,j}=\frac{1}{[x_i-y_j][x_i-y_j-1]}.$$
\end{theorem}

Let $u_i=q^{2x_i}$ and $v_i=q^{2y_i}$ for all $1 \leq i \leq n$. Substituting these into Theorem \ref{ikthv1}, we derive 
\begin{equation}\label{zuvv2}
        Z_n(X,Y)=\frac{(-1)^{\frac{n(n+1)}{2}}\prod_{1 \leq i,j \leq n}(u_i-v_j)(u_i-q^2v_j)
				\times 
\det{\bigl(K_{i,j}\bigr)_{1 \le i,j \le n}}
				}{(q^2-1)^{n^2-n}\prod_{i=1}^{n}(v_iu_i)^{\frac{n}{2}}v_i^{-1}\prod_{1 \leq i <j \leq n}(u_i-u_j)(v_i-v_j)},
\end{equation}
where 
$$K_{i,j}=\frac{1}{(u_i-v_j)(u_i-q^2v_j)}.$$
The above identity will be used in the subsequent proof.
%we obtain the following theorem, which will be used in the subsequent proof.
% \begin{theorem}\label{ikthv2}
%     For a positive integer $n$ and two sets $U = \{u_1, u_2, \dots, u_{n}\}$ and $V = \{v_1, v_2, \dots, v_{n}\}$, the partition function $Z_n(U,V)$ is given by 
%     \begin{equation}\label{zuvv2}
%         Z_n(U,V)=\frac{(-1)^{\frac{n(n+1)}{2}}\prod_{1 \leq i,j \leq n}(u_i-v_j)(u_i-q^2v_j)}{(q^2-1)^{n^2-n}\prod_{i=1}^{n}(v_iu_i)^{\frac{n}{2}}v_i^{-1}\prod_{1 \leq i <j \leq n}(u_i-u_j)(v_i-v_j)}\det{\bigl(K_{i,j}\bigr)_{1 \le i,j \le n}},
%     \end{equation}
%     where the entries $K_{i,j}$ for $1 \leq i, j \leq n$, are given by
%     $$K_{i,j}=\frac{1}{(u_i-v_j)(u_i-q^2v_j)}.$$
% \end{theorem}
% 
% 

% >>>

\section{Lascoux's formula}\label{sec:Lascoux} % <<<

In this section, we present Lascoux’s result on representing symmetric functions as products of two rectangular matrices. 
Let $n$ be a positive integer and let  $X = \{x_1, x_2, \dots, x_{n}\}$ and $Y = \{y_1, y_2, \dots, y_{n}\}$ be two variable sets.
Define
$$R(X,Y):=\prod_{i=1}^{n}\prod_{j=1}^{n}(x_i-y_j)$$
and 
$$
\Delta(X):=\prod_{1 \leq i<j \leq n}(x_i-x_j), \quad \Delta(Y):=\prod_{1 \leq i<j \leq n}(y_i-y_j),
$$
where $\Delta(X)$ and $\Delta(Y)$ are the Vandermonde determinants of $X$ and $Y$, respectively.
The Cauchy determinant formula is given by 
$$\det{\left(\frac{1}{x_i-y_j}\right)}_{1 \le i,j \le n}=(-1)^{\frac{n(n-1)}{2}}\frac{\Delta(X)\Delta(Y)}{R(X,Y)}.$$

Let $e_k$ and $h_k$ denote the elementary and homogeneous symmetric functions, respectively, as in \cite{Macdonald1995}. 
Their generating functions are
$$
	\sum_{k\geq 0} h_k(X) t^k = \frac{1}{\prod_{j=1}^n (1-tx_j)} \quad \text{and}\quad
	\sum_{k\geq 0} e_k(Y) t^k = \prod_{j=1}^n (1+ty_j).
$$
Using divided differences, Lascoux \cite{La99} showed that these functions admit a matrix factorization depending separately on $X$ and $Y$. We now recall this result.

\begin{theorem}[Lascoux]\label{th:lascoux} 
 Let $V = \{v_1, v_2, \dots, v_{n}\}$ and $U = \{u_1, u_2, \dots, u_{n}\}$ be two sets of variables.
Then, the symmetric function 
     $$
     F_q(V,U):=(-1)^{\frac{n(n-1)}{2}}\det{\left(\frac{1}{(v_i-u_j)(qv_i-u_j)}\right)_{1 \le i,j \le n}} \frac{\prod_{1 \leq i,j \leq n}(v_i-u_j)(qv_i-u_j)}{\Delta(V)\Delta(U)}
     $$
     is equal to the determinant of the product of matrices
   $$
     \left(h_{j-i}(V)\right)_{1\leq i \leq n, 1\leq j\leq 2n-1}
     \times
		 \left(\frac{q^{j-k+1}-q^{k-1}}{q-1}(-1)^{n-j+k-1}e_{n-j+k-1}(U)\right)_{1 \leq j \leq 2n-1, 1 \leq k \leq n}.
     $$
 \end{theorem}

Applying Theorem~\ref{th:lascoux} to the identity \eqref{zuvv2} gives the following result, which is crucial for proving our main theorem.

\begin{theorem}\label{zthuv}
Given two sets of parameters $X=\{x_1,\ldots,x_n\}$ and parameters $Y=\{y_1,\ldots,y_n\}$. 
Define $U=\{q^{2x_1}, \ldots, q^{2x_n}\}$ and $V=\{q^{2y_1}, \ldots, q^{2y_n}\}$.
Then, the partition function of the square ice configuration defined in \eqref{def:Zn} is equal to
    \begin{equation*}
        Z_n(X,Y)=\frac{(-1)^{n^2}}{(q^2-1)^{n^2-n}\prod_{i=1}^{n}(v_iu_i)^{\frac{n}{2}}v_i^{-1}}\times F_{q^2}(V,U),
    \end{equation*}
    where $F_{q^2}(V,U)$ is the determinant of the product of matrices
    \begin{equation*}
        \bigl(h_{j-i}(V)\bigr)_{1\leq i \leq n, 1\leq j \leq 2n-1} 
				\times \left(\frac{q^{2j-2k+2}-q^{2k-2}}{q^2-1}
				(-1)^{n-j+k-1}
				e_{n-j+k-1}(U)\right)_{1 \leq j \leq 2n-1, 1 \leq k \leq n}.
    \end{equation*}
\end{theorem}

% >>>

\section{Proof of Theorem~\ref{th:main}}\label{sec:proof_mth} % <<<
Let $\alpha$ and $\beta$ be two parameters. Set 
$$
\rho=-\frac{[\alpha-\frac{1}{2}]}{[\alpha+\frac{1}{2}]} \text{\quad and\quad}
\tau=-\frac{[\beta+\frac{1}{2}]}{[\beta-\frac{1}{2}]}.
$$
Let $n$ be a positive integer.  
Since the expressions we obtain are rather lengthy, we first introduce some notation to simplify the exposition.
\begin{align*}
\mathcal{Z}_n(\frac{1}{2},\alpha, \beta) 
&= Z_n(\{\frac{1}{2} +\alpha, \frac{1}{2}, \frac{1}{2}, \dots, \frac{1}{2}, \frac{1}{2} + \beta\},
\{0, 0,\dots, 0\}),\\
	\mathcal{C}_n(\alpha, \beta)&=
-q^{\alpha+\beta+\frac{n}{2}}[\frac12]^{-{(n-1)(n-2)}}\frac{[\alpha-\frac{1}{2}]}{[\alpha+\frac{1}{2}]^n}\frac{[\beta+\frac{1}{2}]}{[\beta-\frac{1}{2}]^n}.
\end{align*}
\begin{theorem}\label{th:A=CZ} 
	The doubly-refined ASMs generating function \eqref{def:An} satisfies
    \begin{equation}\label{athq}
			A_n([\frac12]^{-2},\rho,\tau)=\mathcal{C}_n(\alpha, \beta) \times \mathcal{Z}_n(\frac{1}{2},\alpha,\beta).
    \end{equation}
\end{theorem}
\begin{proof}
	Suppose $A$ is an $n\times n$ ASM. Then, the weight of the corresponding square ice configuration $\varphi(A)$ is
\begin{align}\label{wa1}
W(\varphi(A)) =&(-q^{-(\xi+\alpha)})(-q^{-(\xi+\beta)})(-q^{-\xi})^{n_1-2}(-q)^{n_2} \nonumber \\
& \times [\xi+\alpha-1]^{f(A)-1}[\xi+\alpha]^{n-f(A)}[\xi+\beta-1]^{n-\ell(A)}[\xi+\beta]^{\ell(A)-1} \nonumber \\
& \times [\xi-1]^{n_3+n_4-(f(A)-1+n-\ell(A))}[\xi]^{n_5+n_6-(\ell(A)-1+n-f(A))}. 
\end{align}
By the bijection $\varphi$, we have $\mu(A)=n_2$. Since $n_1=n+n_2$ and $\sum_{i=1}^6n_i=n^2$, we have
\begin{equation}\label{ni2}
    n_3+n_4+n_5+n_6=n^2-n-2\mu(A).
\end{equation}
Substituting $(\ref{ni})$ and $(\ref{ni2})$ into $(\ref{wa1})$, we obtain 
\begin{align*}
W(\varphi(A)) 
=& (-1)^{f(A)-\ell(A)-1}q^{-(\alpha+\beta)}[\alpha-\frac{1}{2}]^{f(A)-1}[\alpha+\frac{1}{2}]^{n-f(A)}[\beta-\frac{1}{2}]^{n-\ell(A)} \nonumber \\
& \times [\beta+\frac{1}{2}]^{\ell(A)-1}[\frac{1}{2}]^{n^2-3n-2\mu(A)+2}.
\end{align*} 
By the definition of the partition function \eqref{def:Zn}, we obtain
\begin{align*}
\mathcal{Z}_n(\frac{1}{2}, \alpha, \beta)
	&=\sum_{A \in \ASM(n)} W(\varphi(A))\\
	&=(\mathcal{C}_n(\alpha, \beta))^{-1} \times  \sum_{A \in \ASM(n)} 
( -\frac{[\alpha - \frac{1}{2}]}{[\alpha + \frac{1}{2}]})^{f(A)}
( -\frac{[\beta + \frac{1}{2}]}{[\beta - \frac{1}{2}]} )^{\ell(A)}
[ \frac{1}{2} ]^{-2\mu(A)}\\
	&= (\mathcal{C}_n(\alpha, \beta))^{-1} \times A_n( [ \frac{1}{2} ]^{-2}, \rho, \tau).
\end{align*}
This is exactly \eqref{athq}. 
\end{proof}

By Theorem~\ref{th:A=CZ}, we have to compute $\mathcal{Z}_{n}(\frac{1}{2},\alpha,\beta)$ and $\mathcal{C}_n(\alpha, \beta)$ for proving Theorem~\ref{th:main}. 
\begin{lemma}\label{lem:Zn} 
Let $n\geq 2$ be an integer. 
Then,
\begin{equation}\label{znf12}
	\mathcal{Z}_n(\frac{1}{2}, \alpha, \beta)=\frac{(-1)^{\frac{n(n-1)}{2}}q^{n(n-\alpha-\beta+\frac{1}{2})}}{(q^2-1)^{n(n-1)}}\times \det{\bigl(R_{i,j}\bigr)_{1 \le i,j \le n}},
\end{equation}
where 
$$
R_{i,j}=\sum_{k=1}^{2n-1}(-1)^k\frac{q^{k-2j+1}-q^{2j-k-3}}{q^2-1}B_n(i,j;k)
$$
with notation
$$
	B_n(i,j;k)=\binom{n+k-i-1}{k-i}\left[\binom{n-2}{k-j-1}+(q^{2\alpha}+q^{2\beta})\binom{n-2}{k-j}+ q^{2(\alpha+\beta)} \binom{n-2}{k-j+1}  \right].
$$
\end{lemma}

\begin{proof} 
To determine the value of $\mathcal{Z}_n(\frac12, \alpha, \beta)$,
it suffices to choose appropriate special values for the parameters in Theorem \ref{zthuv}.
In this case,
$$
	U=\{q^{2\alpha+1}, q,q,\ldots, q , q^{2\beta +1} \},
	\quad
	V=\{1,1,\ldots, 1 \},
$$
so that
\begin{align*}
	\sum_{k=0}^{\infty}t^ke_k(U)&=\prod_{j=1}^{n}(1+tu_j)=(1+tq^{2\alpha+1})(1+tq^{2\beta+1})(1+tq)^{n-2},\\
	\sum_{k=0}^{\infty}t^kh_k(V)&=(\prod_{j=1}^{n}(1-tv_j))^{-1}=(1-t)^{-n}.
\end{align*}
By extracting the coefficient of $t^k$, we obtain
\begin{align*}
	e_k(U)&=\left[ \binom{n-2}{k}+(q^{2\alpha} +q^{ 2\beta})\binom{n-2}{k-1}+q^{2(\alpha+\beta)} \binom{n-2}{k-2}  \right] q^k,\\
		h_k(V)&=\binom{n+k-1}{k}.
\end{align*}
By substituting the above identities into Theorem~\ref{zthuv}, we conclude the proof.
\end{proof}

\begin{lemma}\label{lem:Cn}
	We have
\begin{align}\label{eq:Cn2}
	\mathcal{C}_n(\alpha, \beta) = [\frac12]^{-{(n-1)(n-2)}}  (-1)^{n} \tau \rho q^{1+n(\alpha+\beta) -\frac{n}{2}}((1+\rho q)(\tau +q))^{n-1}.
\end{align}
\end{lemma}
\begin{proof}
By the definitions of $[\xi]$, we have
\begin{equation*}
    \rho = -\frac{[\alpha-\frac{1}{2}]}{[\alpha+\frac{1}{2}]} 
         = \frac{\phi-q}{1-\phi q} 
    \quad \text{and} \quad 
    \tau = -\frac{[\beta+\frac{1}{2}]}{[\beta-\frac{1}{2}]}
         = \frac{1-\psi q}{\psi - q},
\end{equation*}
where $\phi=q^{2\alpha}$ and $\psi=q^{2\beta}$.  
So that
\begin{equation}\label{ppq}
    \phi=\frac{\rho+q}{1+\rho q} 
    \quad \text{and} \quad 
    \psi=\frac{1+\tau q}{\tau +q}.
\end{equation}
By applying (\ref{ppq}) to $[\alpha+\frac{1}{2}]$, $[\alpha-\frac{1}{2}]$,  $[\beta +\frac{1}{2}]$, and $[\beta -\frac{1}{2}]$, we derive the following:
\begin{align*}
	[\alpha+\frac{1}{2}]&=q^{-\alpha}\frac{\phi q^{\frac{1}{2}}-q^{-\frac{1}{2}}}{q-q^{-1}}=q^{-\alpha}\frac{ \frac{\rho+q}{1+\rho q}  q^{\frac{1}{2}}-q^{-\frac{1}{2}}}{q-q^{-1}}=\frac{q^{\frac{1}{2}-\alpha}}{1+\rho q},\\
	[\alpha-\frac{1}{2}]&=q^{-\alpha}\frac{\phi q^{-\frac{1}{2}}-q^{\frac{1}{2}}}{q-q^{-1}}=q^{-\alpha}\frac{  \frac{\rho+q}{1+\rho q} q^{-\frac{1}{2}}-q^{\frac{1}{2}}}{q-q^{-1}}=-\frac{\rho q^{\frac{1}{2}-\alpha}}{1+\rho q},\\
	[\beta-\frac{1}{2}]&=q^{-\beta}\frac{\psi q^{-\frac{1}{2}}-q^{\frac{1}{2}}}{q-q^{-1}}= q^{-\beta}\frac{ \frac{1+\tau q}{\tau +q}  q^{-\frac{1}{2}}-q^{\frac{1}{2}}}{q-q^{-1}}=-\frac{q^{\frac{1}{2}-\beta}}{\tau+q},\\
	[\beta + \frac{1}{2}]&= q^{-\beta}\frac{\psi q^{\frac{1}{2}}-q^{-\frac{1}{2}}}{q-q^{-1}}= q^{-\beta}\frac{ \frac{1+\tau q}{\tau +q}  q^{\frac{1}{2}}-q^{-\frac{1}{2}}}{q-q^{-1}}=\frac{\tau q^{\frac{1}{2}-\beta}}{\tau+q}.
\end{align*}
	Substituting the above identities into the definition of $\mathcal{C}_n(\alpha, \beta)$, we obtain \eqref{eq:Cn2}.
\end{proof}

Having established these preliminary results, we can now prove Theorem~\ref{th:main}.

\begin{proof}[Proof of Theorem~\ref{th:main}]
Let $z=[\frac{1}{2}]^{-2}=[2]+2+\frac{q^2-q^{-2}}{q-q^{-1}}+2=q+q^{-1}+2$.
From Theorem~\ref{th:A=CZ} and Lemma \ref{lem:Zn},
\begin{equation*}
	A_n(z,\rho,\tau) 
	= \mathcal{C}_n(\alpha, \beta) \times \mathcal{Z}_n(\frac{1}{2},\alpha,\beta)
	= \mathcal{D} \times \det{\bigl(R_{i,j}\bigr)_{1 \le i,j \le n}},
\end{equation*}
where
$$
\mathcal{D} = \mathcal{C}_n(\alpha, \beta) \times \frac{(-1)^{\frac{n(n-1)}{2}}q^{n(n-\alpha-\beta+\frac{1}{2})}}{(q^2-1)^{n(n-1)}}.
$$
We proceed to express $\alpha$ and $\beta$ in term of $\rho$ and $\tau$.
By Lemma \ref{lem:Cn} and the identity $(q^2-1)^2=q^2(q+q^{-1}+2)(q+q^{-1}-2)$, we have
\begin{align*}
	\mathcal{D}
	&= 
	z^{\frac{(n-1)(n-2)}{2}} { (-1)^{n} \tau \rho q^{1+n(\alpha+\beta) -\frac{n}{2}}((1+\rho q)(\tau +q))^{n-1}  }
	\times \frac{(-1)^{\frac{n(n-1)}{2}}q^{n(n-\alpha-\beta+\frac{1}{2})}}{(q^2-1)^{n(n-1)}}\\
	 &= \frac{ (-1)^{\frac{n(n+1)}{2}} \tau \rho q^{n+1}((1+\rho q)(\tau +q))^{n-1}  }{
		 (q+q^{-1}+2 )^{n-1} (q+q^{-1}-2)^{\frac{n(n-1)}{2}}},\\
	 &= \left(\frac{ (1+\rho q)(\tau +q) }{(1+q)^2} \right)^{n-1} \frac{ \tau \rho }{
		 (2-q-q^{-1})^{\frac{n(n-1)}{2}}} \times (-q^2)^n.
\end{align*}
On the other hand, 
	\begin{align*}
		\det(R_{i,j})_{1\leq i, j\leq n} &= \det(-q^{-2} M_{i,j})_{1\leq i,j\leq n},
	\end{align*}
where
\begin{align*}
    M_{i,j}= &\sum_{k=1}^{2n-1}(-1)^{k+1}\frac{q^{k-2j+2}-q^{2j-k-2}}{q-q^{-1}}\binom{n+k-i-1}{k-i}\left[\binom{n-2}{k-j-1} \right.\\
     &\left. +(\frac{\rho +q}{1+\rho q}+\frac{1+\tau q}{\tau +q})\binom{n-2}{k-j}+\frac{(\rho +q)(1+\tau q)  }{(1+\rho q)(\tau +q) } \binom{n-2}{k-j+1}  \right].
\end{align*}
Hence,
	\begin{equation}\label{eq:AnM}
	A_n(z,\rho,\tau) 
	= 
	 \left(\frac{ (1+\rho q)(\tau +q) }{(1+q)^2} \right)^{n-1} \frac{ \tau \rho }{
		 (2-q-q^{-1})^{\frac{n(n-1)}{2}}} 
	 \times \det{\bigl(M_{i,j}\bigr)_{1 \le i,j \le n}}.
\end{equation}

Using Pascal's identity of binomial coefficients
$\binom{n}{k}=\binom{n-1}{k}+\binom{n-1}{k-1}$, we apply the elementary row operations to the matrix $(M_{i,j})_{1 \le i,j \le n}$
	in an appropriate manner. The matrix $(M_{i,j})_{1\leq i, j\leq n}$ can be transformed to 
$(M'_{i,j})_{1\leq i, j\leq n}$, where
\begin{align}\label{mm}
    M'_{i,j} 
      =& \sum_{k=1}^{2n-1}(-1)^{k+1}\frac{q^{k-2j+2}-q^{2j-k-2}}{q-q^{-1}}\binom{k-1}{i-1}\left[\binom{n-2}{k-j-1} \right.\nonumber \\
      &\left. +(\frac{\rho +q}{1+\rho q}+\frac{1+\tau q}{\tau +q})\binom{n-2}{k-j}+\frac{(\rho +q)(1+\tau q)  }{(1+\rho q)(\tau +q) } \binom{n-2}{k-j+1}  \right].
\end{align}
It is straightforward to verify that 
$$\det{\bigl(M_{i,j}\bigr)_{1 \le i,j \le n}}
	=\det{\bigl(M'_{i,j}\bigr)_{1 \le i,j \le n}}.
$$

Substituting $k$ with $k+2j+2$ in (\ref{mm}) and removing all terms that are identically zero in the summation, 
we obtain the matrix $(K^{\rho,\tau}_{i,j}(q))_{1 \le i,j \le n}$ defined in Theorem~\ref{th:main}. So that
$$\det{\bigl(M'_{i,j}\bigr)_{1 \le i,j \le n}}
	=\det{\bigl(K^{\rho,\tau}_{i,j}(q)\bigr)_{1 \le i,j \le n}}.
$$
Identity \eqref{eq:AnM} is exactly \eqref{eq:An:main}.
This completes the proof. 
\end{proof}

% >>>

\section{Applications and Corollaries}\label{sec:applications} % <<<
Our main theorem yields several interesting corollaries concerning matrix determinants. 
Setting $\tau=1$ in Theorem~\ref{th:main} yields the following corollary. 

\begin{corollary}\label{cor:tau=1}
    The generating function of the $n\times n$ ASMs for the number of $-1$ and the position of $1$ in the first row is 
    \begin{equation*}
			A_n(2+q+q^{-1},\rho, 1)=\sum_{A\in \ASM(n)}(2+q+q^{-1})^{\mu(A)}\rho^{f(A)}=\frac{\rho\cdot \det{\bigl(K^{\rho,1}_{i,j}(q)\bigr)_{1 \le i,j \le n}}}{(1+\rho q)(1+q)^{n-1}(2-q-q^{-1})^{\frac{n(n-1)}{2}}},
    \end{equation*}
     where 
$$
K^{\rho,1}_{i,j}(q)
	= \sum_{k=-n+1}^n(-1)^{k+1}\frac{q^k-q^{-k}}{q-q^{-1}}{2j+k-3 \choose i-1} \left[(1+\rho q){n-1 \choose j+k-2}
	+ (\rho +q){n-1 \choose j+k-1}\right].
$$
\end{corollary}

\begin{corollary}\label{cor:JRL}
	For a positive integer $n$, we have
\begin{equation*}
	\det\bigl( J_{i,j}(q) \bigr)_{1 \le i,j \le n+1} 
	= q^{n+2} \det\bigl(  J_{i,j}(q^{-1} )\bigr)_{1 \le i,j \le n+1} 
	= q (1-q)^n(q-q^{-1})^n \det\bigl( L_{i,j}(q) \bigr)_{1 \le i,j \le n},
\end{equation*}
where
\begin{align*}
	J_{i,j}(q) &= \sum_{k=-n}^{n+1} (-1)^{k+1} \frac{q^k - q^{-k}}{q - q^{-1}} 
\binom{2j + k - 3}{i - 1} 
\left[ q \binom{n}{j + k - 2} + \binom{n}{j + k - 1} \right], \\
	L_{i,j}(q) &= \sum_{k=-n+1}^{n} (-1)^{k+1} \frac{q^k - q^{-k}}{q - q^{-1}} 
\binom{2j + k - 3}{i - 1} 
\binom{n}{j + k - 1}.
\end{align*}
\end{corollary}

\begin{proof} By the definition of $A_{n}(2+q+q^{-1},\rho,1)$, we see that the coefficient of $\rho^1$ 
	in $A_{n+1}(2+q+q^{-1},\rho,1)$ is just $A_{n}(2+q+q^{-1},1,1)$. 
	Selecting the coefficients of $\rho^1$ in  Corollary \ref{cor:tau=1} with $n:=n+1$, we have
\begin{equation}\label{eqco1}
	A_{n}(2+q+q^{-1},1,1) = \frac{\det\bigl(  K^{0,1}_{i,j}(q) \bigr)_{1 \le i,j \le n+1}}{(1+q)^n(2-q-q^{-1})^{\frac{n(n+1)}{2}}}
	 = \frac{\det\bigl( q J_{i,j}(q^{-1}) \bigr)_{1 \le i,j \le n+1}}{(1+q)^n(2-q-q^{-1})^{\frac{n(n+1)}{2}}}. 
\end{equation}
Moreover, when $\rho=1$ in Corollary~\ref{cor:tau=1} with $n:=n$, we have
 \begin{equation}\label{eqco2}
    A_{n}(2+q+q^{-1},1,1)
	 = \frac{ \det\bigl( K^{1,1}_{i,j}(q) \bigr)_{1 \le i,j \le n}}{(1+q)^n(2-q-q^{-1})^{\frac{n(n-1)}{2}}}
	 = \frac{\det\bigl( L_{i,j}(q) \bigr)_{1 \le i,j \le n}}{(2-q-q^{-1})^{\frac{n(n-1)}{2}}}.
 \end{equation}
	Combining (\ref{eqco1}) and (\ref{eqco2}), we have the second identity in Corollary \ref{cor:JRL}. 
	Substituting $q$ by $q^{-1}$, we derive the first identity in Corollary \ref{cor:JRL}. 
\end{proof}

From the first identity in Corollary~\ref{cor:JRL}, we derive the following result.
\begin{corollary}\label{cor:sym}
	The coefficients of $q^{-1-\frac{n}{2}}\det{\left(J_{i,j}(q)\right)_{1 \le i,j \le n}}$ are symmetric. 
\end{corollary}
For example, if $n=1$,
$$
q^{-\frac{3}{2}} \det(J_{i,j}(q))_{1\leq i, j\leq 2} = 
-q^{\frac32} + q^{\frac12}
+ q^{-\frac12} -q^{-\frac32};
$$
and if $n=2$,
$$
q^{-2}\det{\left(J_{i,j}(q)\right)_{1 \le i,j \le 3}}= -2q^4+8q^3-8q^2-8q+20-8q^{-1}-8q^{-2}+8q^{-3}-2q^{-4}.
$$
We wonder whether there exists a direct proof of Corollary \ref{cor:sym} using the definitions of $J_{i,j}(q)$.

\medskip

Throughout this paper, we represent the imaginary unit by $I$, where $I=\sqrt{-1}$.
Setting $q=\omega_{-}=({\sqrt{3}I - 1})/{2}$, i.e.,  $2+q+q^{-1}=1$, in 
Identity~\eqref{eqco2}, we get
\begin{equation}\label{eq:A111}
  A_{n}(1,1,1)
	= \frac{\det\bigl( L_{i,j}(\omega_{-}) \bigr)_{1 \le i,j \le n}}{3^{\frac{n(n-1)}{2}}}.
\end{equation}
Let
\begin{equation*}
	\delta_{-}(k) =  (-1)^{k+1}\frac{\omega_{-}^k- \omega_{-}^{-k}}{\omega_{-}- \omega_{-}^{-1}}. 
\end{equation*}
We verify that $\delta_{-}(k+6m) = \delta_{-}(k)$ and
$$
\delta_{-}(k) = 
\begin{cases}
0, & \text{if } k \equiv 0, 3 \pmod{6}; \\
1, & \text{if } k \equiv 1, 2 \pmod{6}; \\
-1, & \text{if } k \equiv 4, 5 \pmod{6}.
\end{cases}
$$
Since $A_n(1,1,1)$ enumerates the $n\times n$ ASMs, it follows from (\ref{asm}) that
\begin{equation}\label{eqc115}
    A_n(1,1,1)=\prod_{k=0}^{n-1}\frac{(3k+1)!}{(n+k)!}.
\end{equation}
Combining Identities \eqref{eq:A111}  and \eqref{eqc115}, we obtain 

\begin{corollary} We have
    \begin{equation}\label{eqc111}
        \det{\bigl(T_{i,j}\bigr)_{1 \le i,j \le n}}=3^{\frac{n(n-1)}{2}}\prod_{k=0}^{n-1}\frac{(3k+1)!}{(n+k)!},
    \end{equation}
 where
    $$
    T_{i,j}=\sum_{k=-n+1}^{n}\delta_{-}(k)\cdot\binom{2j+k-3}{i-1}\binom{n}{j+k-1}.$$
\end{corollary}

Setting $q=I$, i.e., $2+q+q^{-1}= 2$ ,  in Identity~\eqref{eqco2} yields 
\begin{equation}\label{eq:A211}
  A_{n}(2,1,1)
	= \frac{\det\bigl( L_{i,j}(I) \bigr)_{1 \le i,j \le n}}{2^{\frac{n(n-1)}{2}}}.
\end{equation}
Since
$$
(-1)^{k+1} \frac{I^{k}-I^{-k}}{I-I^{-1}}  =
\begin{cases}
	0,      & \text{if $k=2m$ even;} \\
	(-1)^m, & \text{if $k=2m+1$ odd,}
\end{cases}
$$
we have
$$
L_{i,j}(I)=\sum_{k=\lceil  -\frac{n}{2}\rceil }^{ \lfloor \frac{n-1}{2} \rfloor}(-1)^{k}\binom{2j+2k-2}{i-1}\binom{n}{j+2k}.
$$
Moreover, it follows from \cite{MRR2} that the $2$-enumeration of ASMs is given by 
\begin{equation}\label{eqc126}
    A_n(2,1,1)=2^{\frac{n(n-1)}{2}}.
\end{equation} 
Combining Identities \eqref{eq:A211}  and \eqref{eqc126}, we obtain 

\begin{corollary}\label{co12}
	We have
$	\det{(L_{i,j}(I))_{1 \le i,j \le n}}=2^{n(n-1)}$.
\end{corollary}

Setting  $q=\omega_{+}=({\sqrt{3}I + 1})/{2}$, i.e.,  $2+q+q^{-1}=3$, in 
Identity~\eqref{eqco2}, we get
\begin{equation}\label{eq:A311}
  A_{n}(3,1,1)
	= {\det\bigl( L_{i,j}(\omega_{+}) \bigr)_{1 \le i,j \le n}}.
\end{equation}
Let
\begin{equation*}
	\delta_{+}(k) =  (-1)^{k+1}\frac{\omega_{+}^k- \omega_{+}^{-k}}{\omega_{+}- \omega_{+}^{-1}}. 
\end{equation*}
We verify that $\delta_{+}(k+3m) = \delta_{+}(k)$ and
$$
\delta_{+}(k) = 
\begin{cases}
0, & \text{if } k \equiv 0 \pmod{3}; \\
1, & \text{if } k \equiv 1 \pmod{3}; \\
-1, & \text{if } k \equiv 2 \pmod{3}.
\end{cases}
$$

Furthermore, from \cite{Ku96} and \cite{MRR2},  the $3$-enumeration of ASMs is given by 
\begin{equation}\label{eqao3}
    A_{2m+1}(3,1,1)=3^{m(m+1)}\prod_{k=0}^{m-1}\frac{(3k+2)!^2}{(m+k+1)!^2},
\end{equation}
\begin{equation}\label{eqae3}
    A_{2m}(3,1,1)=3^{m^2-1}\frac{(m-1)!}{(3m-1)!}\prod_{k=0}^{m-1}\frac{(3k+2)!^2}{(m+k)!^2},
\end{equation}
for $m \geq 0$. Combining \eqref{eq:A311}, \eqref{eqao3}, and \eqref{eqae3} yields

\begin{corollary} 
We have
\begin{equation*}
\det{\bigl(T_{i,j}\bigr)_{1 \le i,j \le n}}
=
\begin{cases}
\displaystyle{3^{m(m+1)}\prod_{k=0}^{m-1}\frac{(3k+2)!^2}{(m+k+1)!^2}}, &\text{if $n=2m+1$}; \\[15pt]
\displaystyle{ 3^{m^2-1} \frac{(m-1)!}{(3m-1)!}\prod_{k=0}^{m-1}\frac{(3k+2)!^2}{(m+k)!^2}}, &\text{if $n=2m$},
\end{cases}
\end{equation*}
where
$$
T_{i,j}=\sum_{k=-n}^{n}\delta_{+}(k)\binom{2j+k-3}{i-1}\binom{n}{j+k-1}.
$$

\end{corollary}

% >>>

\section{Comparison with known results}\label{sec:comparison} % <<<
To date, numerous results have been obtained on refined enumerations of ASMs. 
Recall the following result due to Aigner \cite{Ai21}. 
\begin{theorem}[Aigner]\label{ai}
    The $(2 + q + q^{-1})$-enumeration of alternating sign matrices is equal to
    \begin{equation*}
			A_n(2+q+q^{-1},1,1) =  \det{\left(\binom{i+j-2}{j-1}\frac{1-(-q)^{j-i+1}}{1+q}\right)_{1 \le i,j \le n}}.
    \end{equation*}
\end{theorem}
Comparing Aigner's Theorem~\ref{ai} and \eqref{eqco2}, we obtain 
\begin{corollary} For any positive integer $n$, the following identity holds:
\begin{equation*}
	 \det{\bigl(L_{i,j}(q)\bigr)_{1 \le i,j \le n}} 
    = {(2-q-q^{-1})^{\frac{n(n-1)}{2}}}
		\det{\left(\binom{i+j-2}{j-1}\frac{1-(-q)^{j-i+1}}{1+q}\right)_{1 \le i,j \le n}},
\end{equation*}
where $L_{i,j}(q)$ is defined in Corollary \ref{cor:JRL}.
\end{corollary}

Setting $q=I$ in Theorem~\ref{ai}, we obtain the following formula for the $2$-enumeration of ASMs:
\begin{equation*}
		A_n(2,1,1) =  \det{\left(\binom{i+j-2}{j-1}\sigma(k)\right)}_{1 \leq i,j \leq n}
\end{equation*}
    with $k=j-i+1$, $\sigma(k+4m)=\sigma(k)$, and 
    $$
\sigma(k) = 
\begin{cases}
0, & \text{if } k \equiv 0 \pmod{4}; \\
1, & \text{if } k \equiv 1 \pmod{4}; \\
1-I, & \text{if } k \equiv 2 \pmod{4}; \\
-I, & \text{if } k \equiv 3 \pmod{4}.
\end{cases}
$$

In what follows, we compare our results with those of others. 
As multiple statistics are involved, we focus on the enumeration according to the number of $-1$ entries for simplicity. 
For convenience, we mainly compare the $2$-enumeration of ASMs.
Let us collect the $2$-enumerations of ASMs by Aigner, Behrend et al.\cite{BEZ13}, and ourselves together in a single corollary.

Let $\Delta_{i,j+1}$ denotes the Kronecker delta, 
  defined by $\Delta_{i,j+1} = 1$ if $i = j+1$, 
  and $\Delta_{i,j+1} = 0$ otherwise.

\begin{corollary}\label{cor:compar} 
For each positive integer $n$, let
\begin{align*}
\mathcal{M}_A &=  {\left(\binom{i+j-2}{j-1}\sigma(j-i+1)\right)}_{1 \leq i,j \leq n},
& \text{(Aigner)}
\\
\mathcal{M}_B &= {\bigl(T_{i,j}\bigr)}_{1 \le i,j \le n},
& \text{(Behrend et al.)}
\\
\mathcal{M}_C &= \left(\sum_{k=\lceil  -\frac{n}{2}\rceil }^{ \lfloor \frac{n-1}{2} \rfloor}
(-1)^{k}2^{1-j}\binom{2j+2k-2}{i-1}\binom{n}{j+2k}\right)_{1\leq i,j\leq n},
& \text{(H. and Y.)}
\end{align*}
where 
$$
 T_{i,j} = -\Delta_{i, j+1} + 
 \begin{cases}
 \displaystyle
 \sum_{k=0}^{\min(i-1, j)} \binom{i-2}{i-k-1} \binom{j}{k} 2^{i-k-1}, & \mbox{if } \,\, j \leq n-1, \\[12pt]
 \displaystyle
 \sum_{k=0}^{i-1} \sum_{\ell=0}^k \binom{i-2}{i-k-1} \binom{n-\ell-1}{k-\ell} 2^{i-k-1}, & \mbox{if } \,\, j = n.
 \end{cases}
$$
Then, 
$$
\det(\mathcal{M}_A)
=\det(\mathcal{M}_B)
=\det(\mathcal{M}_C)=A_n(2,1,1).
$$
\end{corollary}

For example, if $n=4$, we have
\begin{align*}
\mathcal{M}_A &= 
\begin{pmatrix}
1 & 1-I & -I & 0 \\
0 & 2 & 3(1-I) & -4I \\
-I & 0 & 6 & 10(1-I) \\
1-I & -4I & 0 & 20
\end{pmatrix},\\
\mathcal{M}_B &=
\begin{pmatrix}
1 & 1 & 1 & 1 \\
0 & 2 & 3 & 4 \\
2 & 4 & 9 & 14 \\
4 & 12 & 24 & 44
\end{pmatrix},\quad
\mathcal{M}_C =
\begin{pmatrix}
0 & 2 & 0 & -1/2 \\
-8 & 4 & 2 & -2 \\
-4 & 0 & 5 & -5/2 \\
0 & -2 & 4 & -1/2
\end{pmatrix}.
\end{align*}
We verify that $
\det(\mathcal{M}_A)
=\det(\mathcal{M}_B)
=\det(\mathcal{M}_C)
= 64
$.

While their determinants are equal,  the forms of the three matrices 
$\mathcal{M}_A, \mathcal{M}_B, \mathcal{M}_C$ are quite different; 
we don't know how to convert one to another by elementary row and column operations.

% >>>

\section{A decomposition conjecture}\label{sec:conj} % <<<
Write $A_n(z):=A_n(z,1,1)$ for short.
A decomposition conjecture concerning $A_n(z)$ was proposed by Mills, Robbins, and Rumsey \cite{MRR2} and later proved by Kuperberg~\cite{Ku96, Ku02}, as stated below.

\begin{theorem}[Kuperberg]\label{th:decomp}
	There exists a family of polynomials $(B_n(z))_{n\geq 1}$ such that 
    \begin{align*}
        A_{2m+1}(z)&=B_{2m+1}(z)B_{2m+2}(z); \\
        A_{2m}(z)&=2B_{2m}(z)B_{2m+1}(z). 
    \end{align*}
\end{theorem}
Moreover, Mills, Robbins, and Rumsey further conjectured that for $n=2m+1$ odd, 
the polynomial $B_n(z)$ is the $z$-enumeration of horizontally symmetric ASMs of order $n$, 
where the weight of an ASM is $z^k$ if there are $k$ $-1$'s to the top of the middle row, 
i.e.,
$$
B_{2m+1}(z) = \sum_{A \in \HSASM(2m+1)}z^{\frac{\mu(A)- m}{2}}.
$$
In this identity, $\HSASM(n)$ denotes the set of all $n\times n$ horizontally symmetric ASMs.

We observe that this decomposition conjecture can be refined. Write $A_n(z,\rho):=A_n(z,\rho,1)$ and 
define
$$
B_{2m+1}(z,\rho) = \sum_{A \in \HSASM(2m+1)}z^{\frac{\mu(A)- {m}}{2}} \rho^{f(A)-2}.
$$
Note that the middle row of a horizontally symmetric ASM is always of the form $(1,-1,1,-1,\dots,1)$, and the entry $1$ in the first row cannot appear in either the first or the last column. It follows that $2 \leq f(A) \leq n-1$.
Therefore, we propose the following conjecture.

\begin{conjecture}
	There exists a family of polynomials $(B_{2m}(z,\rho))_{m\geq 1}$ in $\{z,\rho\}$ such that,
	for $m\geq 1$,
    \begin{align*}
			 A_{2m+1}(z,\rho)&=\rho B_{2m+1}(z,1)B_{2m+2}(z,\rho),\\
			  A_{2m}(z,\rho)&=\rho(\rho+1)B_{2m}(z,1)B_{2m+1}(z,\rho).
    \end{align*}
\end{conjecture}

We verify the above conjecture for small integers $n$. 
By Theorem \ref{th:decomp}, we have
$$
B_2(z,1)= B_3(z,1)=1, \ 
B_4(z,1)=z + 6, \ 
B_5(z,1)= z + 2.
$$
By the definition of $B_{2m+1}(z,\rho)$, we have
$$
B_3(z,\rho)=1, \ 
B_5(z,\rho)=\rho^2 + \rho z + 1.
$$
By the definitions of $A_n(z,\rho)$, we have
\begin{align*}
	A_2(z, \rho)&=\rho(\rho+1) = \rho(\rho+1) B_2(z,1) B_3(z,\rho) ,\\
	A_3(z, \rho)&=\rho \cdot (2\rho^2+\rho z+2 \rho +2) = \rho B_3(z,1) B_4(z,\rho),\\
	A_4(z,\rho) &= \rho (\rho + 1) (z + 6) (\rho^2 + \rho z + 1)   = \rho(\rho+1) B_4(z,1) B_5(z,\rho),\\
	A_5(z, \rho) &= \rho B_5(z,1) B_6(z,\rho),
\end{align*}
where 
\begin{align*}
B_4(z,\rho)&=2\rho^2+\rho z+2 \rho +2,\\
B_6(z,\rho)&= z^3 \rho^2 + 3z^2 \rho^3 + 2z\rho^4 + 6z^2 \rho^2 
	 + 20z\rho^3 + 12 \rho^4 + 3z^2 \rho + 26 z \rho^2 + 12 \rho^3 \\
	 &\qquad + 20z \rho + 12\rho^2 + 2z + 12\rho + 12.
\end{align*}
The conjecture has been verified for $A_n(z,\rho)$ for all $n \leq 9$ by computer.

% >>>

\section*{Acknowledgments}
The second author was supported by the China Scholarship Council.

% >>>

\end{document}